\documentclass[12pt,a4paper,reqno,pdftex]{amsart}
\usepackage[top=30truemm,bottom=30truemm,left=30truemm,right=30truemm]{geometry} 
\usepackage{amssymb}
\usepackage{amsthm}
\usepackage{amsmath}
\usepackage{color}
\usepackage{amsrefs}
\usepackage{cleveref}
\usepackage{tikz}
\usepackage{cases}
\usepackage{mathtools}
\usepackage{empheq}
\usepackage[framemethod=tikz]{mdframed}
\usepackage{stmaryrd}
\usepackage{ulem, datetime}

\usepackage{lineno}
\numberwithin{equation}{section}

\theoremstyle{definition}
 \newtheorem{theorem}{Theorem}[section]
 \crefname{theorem}{Theorem}{Theorems}
 \newtheorem{proposition}[theorem]{Proposition}
 \crefname{proposition}{Proposition}{Propositions}
 \newtheorem{lemma}[theorem]{Lemma}
 \crefname{lemma}{Lemma}{Lemmas}
 \newtheorem{corollary}[theorem]{Corollary}
 \crefname{corollary}{Corollary}{Corollaries}
 
 \crefname{conjecture}{Conjecture}{Conjectures}
 
 \crefname{question}{Question}{Questions}
 
 \crefname{problem}{Problem}{Problems}
 \newtheorem{remark}[theorem]{Remark}
 \crefname{remark}{Remark}{Remarks}

\theoremstyle{definition} 
 \newtheorem{definition}[theorem]{Definition}
 \crefname{definition}{Definition}{Definitions}
 \newtheorem{example}[theorem]{Example}
 \crefname{example}{Example}{Examples}
 
 \crefname{caution}{Caution}{Cautions}
 
 \crefname{equation}{formula}{formulas}

\newcommand{\N}{\text{Nr}}

\newcommand{\syl}{\textup{Syl}}

\newcommand{\C}{\mathbb{C}}

\newcommand{\F}{\mathbb{F}}

\newcommand{\Q}{\mathbb{Q}}

\newcommand{\Z}{\mathbb{Z}}
\newcommand{\sym}{\mathfrak{S}}

\newcommand{\e}{\epsilon}
\newcommand{\s}{\sigma}

\newcommand{\sgn}{\textup{sgn}}

\DeclareMathOperator{\Gal}{Gal} %Galois group
\DeclareMathOperator{\ord}{ord} %order
\DeclareMathOperator{\rank}{rank} %rank
\DeclareMathOperator{\Res}{Res} %Weil restriction
\title{The $p$-adic limits of iterated $p$-power cyclic resultants of multivariable polynomials{\footnote{\today, \number\time}}}
     
\author{Hyuga Yoshizaki} 
\email{yoshizaki.hyuga@gmail.com}
\address{Department of Mathematics, Faculty of Science and Technology, Tokyo University of Science; 2641 Yamazaki, Noda-shi, Chiba, Japan}
\dedicatory{To the memory of Valon}

\begin{document}

\maketitle

\begin{abstract}
    Let $p$ be a prime number.
    The $p$-power cyclic resultant of a polynomial is the determinant of the Sylvester matrix of $t^{p^n}-1$ and the polynomial.
    It is known that the sequence of $p$-power cyclic resultants and its non-$p$-parts converge in $\Z_p$.
    This article shows the $p$-adic convergence of the iterated $p$-power cyclic resultants of multivariable polynomials.
    As an application, we show the $p$-adic convergence of the torsion numbers of $\Z_p^d$-coverings of links.
    We also explicitly calculate the $p$-adic limits for the twisted Whitehead links as concrete examples.
    Moreover, in a specific case, we show that our $p$-adic limit of torsion numbers coincides with the $p$-adic torsion, which is a homotopy invariant of a CW-complex introduced by S. Kionke.
\end{abstract}
%--------------------------------------------------------------------------------------
\section{Introduction}
The cyclic resultants of polynomials appear naturally in various fields of mathematics, such as low dimensional topology and number theory.
For example, in low dimensional topology, Fox--Weber's formula asserts that the torsion number (the size of the torsion subgroup of the first homology group) of a cyclic covering of a knot can be calculated by the cyclic resultant of the Alexander polynomial of the knot (cf. \cite{CWeber1979}).
In number theory, the result analogous to the above holds for global function fields, that is, the class number of a constant extension of a global function field can be calculated by the cyclic resultant of the zeta polynomial of the curve corresponding to the function field (cf. \cite{Rosen2002GTM}*{Corollary to Proposition 8.16} or \cite{UekiYoshizaki-plimits}*{Proposition 7.2}).
As a sequence, cyclic resultants of polynomials coincide with the Lehmer--Pierce sequence, which has long been studied from the perspective of elementary number theory (cf. \cite{Flatters2009}, \cite{Lehmer1933}, \cite{Pierce1916}, and \cite{SKALBA2018}).

Let $p$ be a prime number.
Recently, Ueki and the author showed the $p$-adic convergence of $p$-power cyclic resultants of polynomials and gave an explicit formula for the $p$-adic limit values (\cite{UekiYoshizaki-plimits}*{Theorem 5.3 and Theorem 5.7}).
This study leads us to the numerical study of the $p$-adic limits of the torsion numbers (resp. the class numbers) of $\Z_p$-coverings of knots (resp. constant $\Z_p$-extensions of global function fields).
This article generalizes the former paper to multivariable polynomials.
Let $d$ be a positive integer.
For a polynomial $f \in \Z[t_1,...,t_d]$ of $d$ variables with integral coefficients and positive integers $n_1,...,n_d$, we write
\begin{align*}
    r_{n_1,...,n_d}(f)=\Res(t_1^{p^{n_1}}-1,...,\Res(t_d^{p^{n_d}}-1,f)),
\end{align*}
where the definition of $\Res$ is in \cref{ResDef}.
This article will establish the following.
\newtheorem*{introtheorem1}{Theorem~\ref{multi_p_conv}}
\begin{introtheorem1}
    For all $f\in\Z[t_1,...,t_d]$,
    the sequence $(r_{n_1,...,n_d}(f))$ and its non-$p$-part $(r_{n_1,...,n_d}(f)_{\text{non-}p})$ converge in $\Z_p$.
\end{introtheorem1}

Mayberry--Murasugi~\cite{Mayberry-Murasugi} and Porti~\cite{Porti2004} showed that the size of the first homology group of an abelian covering branched along a link $L$ can be calculated by the multivariable Alexander polynomials of sublinks of $L$.
By combining their results with \cref{multi_p_conv}, we show the $p$-adic convergence of the sizes of the first homology groups in $\Z_p^d$-coverings of links (see \cref{defiOfZpd} for the definition of $\Z_p^d$-coverings).
\newtheorem*{introtheorem2}{Theorem~\ref{convergenceHom}}
\begin{introtheorem2}
    Let $M$ be a closed integral homology $3$-sphere and $L$ be a link in $M$.
    % Let $L\subset M$ be a link and $(M_n \to M)$ be a $\Z_p^d$-covering of $M$ branched along $L$. 
    Suppose that for each finite index open subgroup $\Gamma\subset \Z_p^{\,d}$, $M_{\Gamma}$ is a rational homology $3$-sphere, that is, $H_1(M_{\Gamma};\Z)$ is a finite group.
    Then for each descending sequence of finite index open subgroups $\Gamma_1\supset\Gamma_2\supset \cdots$ in $\Z_p^{\,d}$ such that $\cap_{n}\Gamma_n=\{0\}$, the non-$p$-parts of $|H_1(M_{\Gamma_n};\Z)|$ converge in $\Z_p$.
    Moreover, the limit value does not depend on the choice of a sequence of finite index open subgroups.
\end{introtheorem2}
\begin{remark}
    We note that Kionke \cite{Kionke2020} defined a new topological invariant, the (relative) $p$-adic torsion, for a finite CW-complex.
    In a specific case, our limit values in \cref{convergenceHom} coincide with the Kionke's relative $p$-adic torsion (see \cref{p_adic_torsion}).
\end{remark}
\begin{remark}
    The $p$-adic convergence of class numbers in $\Z_p$-extensions of number fields was initially announced by Sinnott in 1985 and published by Han \cite{Han1991}.
    Its generalization and a variant for function fields were established by Kisilevsky using cyclic resultants \cite{Kisilevsky1997PJM}.
    Our previous result in \cite{UekiYoshizaki-plimits} may be seen as a generalization of Kisilevsky's result. 
    We remark that Ozaki proved the $p$-adic convergence of the class numbers and other invariants in a general pro-$p$ extension of a number field, by using analytical method \cite{Ozaki-padiclimit}.
    He also proved a similar result to \cref{convergenceHom} \cite[Theorem 2]{Ozaki-padiclimit}.
    % We note that \cref{convergenceHom} also follows from Ozaki's result \cite[Theorem 2]{Ozaki-padiclimit} and our previous result \cite[Proposition 4.1, Proof 2]{UekiYoshizaki-plimits}.
\end{remark} 

The contents of this article are as follows.
In Section 2, we recall the one variable version and prepare a general theory of multiple sequences (\cref{CauchyConverge} and \cref{multi_sequentially_limit}).
In Section 3, we show the main theorem (\cref{multi_p_conv}) and give a somewhat general example (\cref{example}).
In Section 4, we show the $p$-adic convergence of the torsion numbers of $\Z_p^d$-coverings of links (\cref{convergenceHom}), and also show that the $p$-adic limit of the torsion numbers coincides with the $p$-adic torsion in a specific case (\cref{p_adic_torsion}).
Finally, we explicitly calculate the $p$-adic limits for the twisted Whitehead links (\cref{twistedWhiteheadLink}).

%--------------------------------------------------------------------------------------
\section{Preliminaries}
In this section, we prepare some terminology and results we will use.
\subsection{Resultant of polynomials}
First, we recall the definition of the resultant of polynomials.
Let $R$ be an integral domain and $R[t]$ denote the ring of polynomials of variable $t$.
For $f$, $g \in R[t]$, write
\begin{align*}
    f&=a_mt^m+a_{m-1}t^{m-1}+ \cdots + a_0, \\
    g&=b_nt^n+b_{n-1}t^{n-1}+ \cdots + b_0,
\end{align*}
where $a_m$, $b_n\neq 0$.
The Sylvester matrix of $f$ and $g$ is defined by
\begin{align*}
    \syl(f,g)=
    \begin{bmatrix}
        a_m & a_{m-1} & \cdots & a_0   &        &      \\
            & \ddots  & \ddots &        & \ddots &      \\
            &         &   a_m  &a_{m-1} & \cdots& a_0  \\
        b_n & b_{n-1} & \cdots & b_0   &        &      \\
            & \ddots  & \ddots &        & \ddots &      \\
            &         &   b_n  &b_{n-1} & \cdots& b_0
    \end{bmatrix},
\end{align*}
where the number of rows that contain $a_m, ..., a_0$ (resp. $b_n,...,b_0$) is $n$ (resp. $m$), and the remaining elements are $0$.
\begin{definition}\label{ResDef}
    The resultant $\Res(f,g)$ of $f$ and $g$ is the determinant of the Sylvester matrix;
    \begin{align*}
        \Res(f,g)=\det\syl(f,g).
    \end{align*}   
\end{definition}
Let $\text{Frac}(R)$ denote the quotient field of $R$.
We note that there is an equivalent definition of the resultant;
    \begin{align*}
        \Res(f,g)=a_m^nb_n^m\prod_{i,j}(\alpha_i-\beta_j),
    \end{align*}
where $\alpha_i$ ($i=1,...,m$) are the roots of $f$ and $\beta_j$ ($j=1,...,n$) are the roots of $g$ in an algebraic closure $\overline{\text{Frac}(R)}$.
Therefore, we have
\begin{align}\label{exchange_law}
    \Res(f,g)=a_m^n\prod_{i}g(\alpha_i)=(-1)^{mn}b_n^m\prod_{j}f(\beta_j).
\end{align}
We also note that if we take the ring of polynomials as $R$, we can define the resultant of multivariable polynomials.
For example, we set $R'=R[t_1]$ and consider the ring of polynomials $R'[t_2]$, where $t_1$ and $t_2$ are algebraically independent variables over $R$.
For $f$, $g\in R'[t_2]$, $\Res(f,g)$ is a polynomial in $R[t_1]$.
% In such a case, we write $\Res_{t_2}(f,g)$ instead of $\Res(f,g)$ to clarify the variable under consideration.

\subsection{$p$-adic limits of cyclic resultants}
Next, we introduce the $p$-adic convergence of the $p$-power cyclic resultants of polynomials and summarize the proof.
Let $\C_p$ denote the $p$-adic completion of an algebraic closure of the $p$-adic number field $\Q_p$.
Let $|\cdot|_p$ denote the $p$-adic absolute value on $\C_p$ normalized so that $|p|_p=1/p$.
For an integer $m$, $m_{\text{non-}p}:=m|m|_p$ denote the non-$p$-part of $m$.
\begin{theorem}[{\cite{UekiYoshizaki-plimits}*{Theorem 5.3}}]\label{p_limits}
    Let $f \in \Z[t]$.
    Then the cyclic resultants $\Res(t^{p^n}-1,f)$ and their non-$p$-parts $\Res(t^{p^n}-1,f)_{\text{non-}p}$ converge in $\Z_p$.
    Moreover, the limit value is $0$ if and only if $f(1)\equiv 0 \mod p$.
\end{theorem}
\begin{proof}
    Since $\Z_p$ is complete, it suffices to show that $\Res(t^{p^n}-1,f)$ is a Cauchy sequence.
    Since $\Res(t^{p^n}-1,f)=\prod_{\zeta^{p^n}=1}f(\zeta)$,
    we have
    \begin{align}\label{fracOfRes}
        \frac{\Res(t^{p^n}-1,f)}{\Res(t^{p^{n-1}}-1,f)}=\prod_{\substack{\zeta^{p^n}=1,\\ \zeta^{p^{n-1}}\neq 1}} f(\zeta).
    \end{align}
    If we write $\zeta_{p^n}=\exp(2\pi\sqrt{-1}/p^n)$, the right hand side of (\ref{fracOfRes}) is written by $\N_{\Q(\zeta_{p^n})/\Q}(f(\zeta_{p^n}))$.
    By the Artin's reciprocity law with modulus $\mathfrak{m}=p^n\infty$;
    \begin{align*}
        I_{\mathfrak{m}}(\Q)/(P_{\mathfrak{m}}(\Q)\N_{\Q(\zeta_{p^n})/\Q}(I_{\mathfrak{m}}(\Q(\zeta_{p^n}))))\cong \Gal(\Q(\zeta_{p^n})/\Q)\cong (\Z/p^n\Z)^*,
    \end{align*}
    we see that $\N_{\Q(\zeta_{p^n})/\Q}(f(\zeta_{p^n})) \equiv 1 \mod {p^n}$.
    By considering the prime ideal factorization of $(f(\zeta_{p^n}))$ in $\Q(\zeta_{p^n})$, a similar argument shows that the sequence of $l$-parts of $\Res(t^{p^n}-1,f)$, write $\Res(t^{p^n}-1,f)_l$, also converges in $\Z_p$ for each prime number $l$.

    Next, we give an equivalent condition for $\lim_{n\to \infty} \Res(t^{p^n}-1,f)=0$.
    If $\lim_{n\to \infty}\Res(t^{p^n}-1,f)=0$ in $\Z_p$, by (\ref{fracOfRes}), there exists $n>0$ such that $p\mid \N_{\Q(\zeta_{p^n})/\Q}(f(\zeta_{p^n}))$.
    This means that $(1-\zeta_{p^n}) \mid f(\zeta_{p^n})$ in $\Z[\zeta_{p^n}]$ and $f(t)=(1-t)g(t)+\Phi_{p^n}(t)h(t)$ for some $g$, $h \in \Z[t]$ where $\Phi_m(t)$ denote the $m$-th cyclotomic polynomial.
    Then we have $f(1)\equiv 0 \mod p$.
    Conversely, if $f(1) \equiv 0 \mod p$, we have $f(t)\equiv (1-t)g(t) \mod p$ for some $g\in \Z[t]$, and $f(t)=(1-t)g(t) + ph(t)$ for some $h\in \Z[t]$.
    For all $n>0$, since $(1-\zeta_{p^n})\mid p$ in $\Z[\zeta_{p^n}]$, we see that $(1-\zeta_{p^n}) \mid f(\zeta_{p^n})$.
    This means that $p\mid \N_{\Q(\zeta_{p^n})/\Q}(f(\zeta_{p^n}))$ for all $n>0$ and the limit value is $0$.
\end{proof}
\begin{remark}\label{0_conditions}
    By the proof of \cref{p_limits}, we see that the following are equivalent;
    \begin{itemize}
        \item $f(1)\equiv 0 \mod p$.
        \item There exists $n>0$ such that $\Res(t^{p^n}-1,f)\equiv 0 \mod p$.
        \item For all $n\geq0$ we have $\Res(t^{p^n}-1,f)\equiv 0 \mod p$.
    \end{itemize}
\end{remark}
Moreover, we have an analogue of Iwasawa type formula, that is, there are three invariants which express the growth rate of the $p$-parts of cyclic resultants.
\begin{proposition}[{\cite{UekiYoshizaki-plimits}*{Part of Theorem 5.7}}]\label{IwasawaTypeFormula}
    For $f\in \Z[t]$, there are unique $\lambda, \mu\in \Z_{\geq 0}$ and $\nu \in \Z$ such that
    \begin{align*}
        |\Res(t^{p^n}-1,f)|_p^{-1}=p^{\lambda n+ \mu p^n +\nu}
    \end{align*}
    for all sufficiently large $n$.
\end{proposition}
Here we give a different proof from \cite{UekiYoshizaki-plimits} in order to make this article self contained.
\begin{proof}
    Let $v_p$ denote the $p$-adic valuation on $\C_p$, and set $e_n(f)=v_p(\Res(t^{p^n}-1,f))$ for $f \in \Z[t]$.
    Let $\mu$ be the maximal integer satisfies that $p^{\mu}$ divides $f$.
    Then we have $\Res(t^{p^n}-1,f)=p^{\mu p^n}\Res(t^{p^n}-1,f/p^{\mu})$.
    Set $g=f/p^{\mu}$.
    Let $a$ be the coefficient of the leading term of $g$.
    Then we have
    \begin{align*}
        e_n(g)=v_p(a^{p^n}\prod_{|\alpha|_p>1}(\alpha^{p^n}-1)\prod_{|\alpha|_p=1}(\alpha^{p^n}-1)\prod_{|\alpha|_p>1}(\alpha^{p^n}-1)),
    \end{align*}
    where $\alpha$ runs the roots of $g$.
    It is easy to see that $v_p(\prod_{|\alpha|_p<1}(\alpha^{p^n}-1))=0$.
    From the Newton polygon of $g$, it follows that $v_p(a\prod_{|\alpha|_p>1}\alpha)=0$.
    Since $v_p(\alpha)=v_p(\alpha-1)$ for $|\alpha|_p>1$ by the strong triangle inequality, we have that $v_p(a^{p^n}\prod_{|\alpha|_p>1}(\alpha^{p^n}-1))=0$.
    If $|\alpha-1|_p=1$ then $|\alpha^{p^n}-1|_p=1$.
    Combining the above calculations for the $p$-adic valuations, we have
    \begin{align*}
        e_n(g)=v_p(\prod_{\substack{|\alpha|_p=1, \\ |\alpha-1|_p<1}}(\alpha^{p^n}-1)).
    \end{align*}
    Here we note that since $\Res(t^{p^n}-1,g)_{\text{non-}p}$ converge, the sequence $(\prod_{|\alpha|_p=1,|\alpha-1|_p<1}(\alpha^{p^n}-1))/p^{e_n(g)}$ converges in $\C_p$ and the limit value is not $0$.
    Thus, by setting $\lambda=\#\{\alpha \mid g(\alpha)=0,\ |\alpha|_p=1,\ |\alpha-1|_p<1\}$, we have
    \begin{align*}
        \lim_{n\to \infty}\frac{\prod_{|\alpha|_p=1,|\alpha-1|_p<1}(\alpha^{p^n}-1)}{p^{e_n(g)}}&=\lim_{n\to\infty}\frac{p^{\lambda n}}{p^{e_n(g)}}\frac{\prod_{|\alpha|_p=1,|\alpha-1|_p<1}(\alpha^{p^n}-1)}{p^{\lambda n}} \\
        &=\left(\prod_{\substack{|\alpha|_p=1, \\ |\alpha-1|_p<1}}\log_p \alpha \right)\lim_{n\to\infty}\frac{p^{\lambda n}}{p^{e_n(g)}},
    \end{align*}
    where $\log_p$ denote the $p$-adic logarithm.

    By setting $\nu=v_p(\prod_{|\alpha|_p=1,|\alpha-1|_p<1}\log_p \alpha)$, we have $v_p(\lim_{n\to\infty}p^{\lambda n+\nu -e_n(g)})=0$.
    This means that $e_n(g)=\lambda n+\nu$ for sufficiently large $n$ and the assertion follows.

\end{proof}
\subsection{Multiple sequences}\label{multi_secuences}
Finally we summarize the theory of multiple sequences.
Let $d$ be a positive integer.
A $d$-tuple multiple integer sequence $(a_{n_1,...,n_d})$ is a function on $\Z_{>0}^d$ to $\Z$ and the image of $(n_1,...,n_d) \in \Z_{>0}^d$ is written by $a_{n_1,...,n_d}$.
\begin{definition}
    A $d$-tuple multiple sequence $(a_{n_1,n_2,...,n_d})$ converges to $\alpha$ in a metric $|\cdot|$, if for all $\e>0$ there exists $N>0$ such that
    for all $n_1, ..., n_d>N$ we have $|a_{n_1,...,n_d}-\alpha|<\e$.
    We denote it as $\lim_{n_1,...,n_d\to\infty}a_{n_1,...,n_d}=\alpha$.
\end{definition}
A multiple sequence $(a_{n_1,...,n_d})$ is called Cauchy sequence if for all $\e>0$, there exists $N>0$ such that for each $m_i>n_i>N$ ($i=1,...,d$), we have $|a_{m_1,...,m_d}-a_{n_1,...,n_d}|<\e$.
Here we recall some important properties of multiple sequences.
\begin{proposition}\label{CauchyConverge}
    In a complete space, a multiple sequence is a convergent sequence if and only if it is a Cauchy sequence.
\end{proposition}
\begin{proof}
    It is obvious that a multiple convergent sequence is a Cauchy sequence.
    Let $(a_{n_1,...,n_d})$ be a multiple Cauchy sequence.
    For each $n>0$, put $b_n=a_{n,...,n}$.
    We see that for all $\e>0$ there exists $N>0$ such that if $n_1,...,n_d>n>N$ then $|a_{n_1,...,n_d}-b_n|<\e$.
    We note that $(b_n)$ is a usual Cauchy sequence, and there exists the limit value $\lim b_n=\alpha$.
    Thus for the same $\e>0$ as above, there exists $M>0$ such that if $n>M$ then $|b_n-\alpha|<\e$.
    Therefore, if $n_1,...,n_d>n>\max \{N,M\}$ then $|a_{n_1,...,n_d}-\alpha|\leq |a_{n_1,...,n_d}-b_n|+|b_n-\alpha|<2\e$, and $(a_{n_1,...,n_d})$ is a convergent sequence. 
\end{proof}
For a positive integer $d$, let $\sym_d$ denote the symmetric group of degree $d$.
\begin{proposition}\label{multi_sequentially_limit}
    Let $(a_{n_1,...,n_d})$ be a convergent multiple sequence and $\s \in \sym_d$.
    For all $i \in \{1, ..., d-1\}$ we suppose that 
    \begin{align*}
        \lim_{n_{\s(i+1)},...,n_{\s(d)}\to\infty}a_{n_1,...,n_d}
    \end{align*}
    converge for each $n_{\s(1)}$, ..., $n_{\s(i)}>0$.
    Then the limit value
    \begin{align*}
        \lim_{n_{\s(1)}\to \infty}\cdots\lim_{n_{\s(d)}\to\infty}a_{n_1,...,n_d}
    \end{align*}
    exists and the value satisfies
    \begin{align*}
        \lim_{n_1,...,n_d\to\infty}a_{n_1,...,n_d}=\lim_{n_{\s(1)}\to \infty}\cdots\lim_{n_{\s(d)}\to\infty}a_{n_1,...,n_d}.
    \end{align*}
\end{proposition}
\begin{proof}
    For each $i\in\{1,...,d-1\}$ and $n_{\s(1)}$,..., $n_{\s(i)}$ we set
    \begin{align*}
        b_{n_{\s(1)},...,n_{\s(i)}}&=\lim_{n_{\s(i+1)}\to\infty}\cdots\lim_{n_{\s(d)}\to\infty}a_{n_1,...,n_d}, \\
        \alpha_{n_{\s(1)},...,n_{\s(i)}}&=\lim_{n_{\s(i+1)},...,n_{\s(d)}\to\infty}a_{n_1,...,n_d}.
    \end{align*}
    We note that the existence of $b_{n_{\s(1)},...,n_{\s(i)}}$ follows from the argument below.
    We show that $b_{n_{\s(1)},...,n_{\s(i)}} = \alpha_{n_{\s(1)},...,n_{\s(i)}}$ by induction on $i$.
    By definition we have $b_{n_{\s(1)},...,n_{\s(d-1)}}=\alpha_{n_{\s(1)},...,n_{\s(d-1)}}$.
    Assume that $b_{n_{\s(1)},...,n_{\s(i)}}=\alpha_{n_{\s(1)},...,n_{\s(i)}}$.
    We show that
    \begin{align*}
        \lim_{n_{\s(i)}\to \infty}b_{n_{\s(1)},...,n_{\s(i)}}=\alpha_{n_{\s(1)},...,n_{\s(i-1)}}
    \end{align*}
    for each $n_{\s(1)}, ..., n_{\s(i-2)}$.
    By the assumption, for all $\e>0$ there exists $N>0$ such that if $n_{\s(i+1)}, ..., n_{\s(d)}>N$ then we have $|a_{n_1,...,n_d}-b_{n_{\s(1)},...,n_{\s(i)}}|=|a_{n_1,...,n_d}-\alpha_{n_{\s(1)},...,n_{\s(i)}}|<\e$.
    On the other hand, for the same $\e>0$ there exists $M>0$ such that if $n_{\s(i)}, ..., n_{\s(d)}>M$ then we have $|a_{n_1,...,n_d}-\alpha_{n_{\s(1)},...,n_{\s(i-1)}}|<\e$.
    Therefore if $n_{\s(i)}>M$ and take $n_{\s(i+1)}, ..., n_{\s(d)}>\max\{M,N\}$, we have $|b_{n_{\s(1)},...,n_{\s(i)}}-\alpha_{n_{\s(1)},...,n_{\s(i-1)}}|\leq|b_{n_{\s(1)},...,n_{\s(i)}}-a_{n_1,...,n_d}|+|a_{n_1,...,n_d}-\alpha_{n_{\s(1)},...,n_{\s(i-1)}}|<2\e$.
    Thus we have $\lim_{n_{\s(i)}\to\infty}b_{n_{\s(1)},...,n_{\s(i)}}=\alpha_{n_{\s(1)},...,n_{\s(i-1)}}$.
    
    It also holds that $\lim_{n_{\s(1)}\to\infty} b_{n_{\s(1)}}=\lim_{n_{\s(1)},...,n_{\s(d)}\to\infty}a_{n_1,...,n_d}$ from $b_{n_{\s(1)}}=\alpha_{n_{\s(1)}}$ by the same argument above, and the assertion follows.
\end{proof}
%--------------------------------------------------------------------------------------
\section{The $p$-adic convergence of multiple cyclic resultants}
For $f(t_1,...,t_d)\in \Z[t_1,...,t_d]$ and $n_1,...,n_d \in \Z_{>0}$, we write
\begin{align*}
    r_{n_1,...,n_d}(f)=\Res(t_1^{p^{n_1}}-1,...,\Res(t_d^{p^{n_d}}-1,f(t_1,...,t_d))).
\end{align*}

First we generalize \cref{0_conditions}.
Set $\mu_p=\{\zeta \in \overline{\Q}_p\mid \exists N\in\Z \text{ such that }\zeta^{p^N}=1\}$.
\begin{lemma}\label{multi_0_conditins}
     Let $f(t_1,...,t_d)\in\Z[t_1,...,t_d]$.
     Then the following hold.
     \begin{itemize}
        \item[(1)] For all $\zeta_1, ..., \zeta_d \in \mu_p$ we have $|f(\zeta_1,...,\zeta_d)-f(1,...,1)|_p<1$.
        \item[(2)] The following are equivalent;
        \begin{itemize}
            \item[\labelitemi] $f(1,...,1)\equiv 0 \mod p$.
            \item[\labelitemi] There exists $n_1, ..., n_d>0$ such that $r_{n_1,...,n_d}(f)\equiv 0 \mod p$.
            \item[\labelitemi] For all $n_1, ..., n_d>0$ we have $r_{n_1,...,n_d}(f)\equiv 0 \mod p$.
        \end{itemize}    
     \end{itemize}
\end{lemma}
\begin{proof}
        Let $at_1^{m_1}\cdots t_d^{m_d}$ be one of the terms of $f$.
        Since $\zeta_1^{m_1}\cdots \zeta_d^{m_d}$ is a $p$-power-th root of unity, we see that $|a\zeta_1^{m_1}\cdots\zeta_d^{m_d}-a|_p<1$.
        Thus we have $|f(\zeta_1,...,\zeta_d)-f(1,...,1)|_p<1$ and the assertion of (1) follows.

        For each $i=1, ..., d-1$ and positive integers $n_{i+1}, ..., n_d$, define
        \begin{align*}
            f_{n_{i+1},...,n_d}(t_i)=\Res(t_{i+1}^{p^{n_{i+1}}}-1,...,\Res(t_d^{p^{n_d}}-1,f(1,...,1,t_i,t_{i+1},...,t_d))).
        \end{align*}
        We note that $r_{n_1,...,n_d}(f)=\Res(t_1^{p^{n_1}}-1,f_{n_2,...,n_d}(t_1))$.
        By using \cref{0_conditions} several times we have
        \begin{align*}
            &f(1,...,1)\equiv 0 \mod p \\
            \Leftrightarrow&\ f_{n_d}(1)\equiv 0 \mod p \\
            \Leftrightarrow&\ f_{n_{d-1},n_d}(1)\equiv 0 \mod p \\
            \cdots& \\
            \Leftrightarrow&\ f_{n_2,...,n_d}(1)\equiv 0 \mod p \\
            \Leftrightarrow&\ r_{n_1,...,n_d}(f) \equiv 0 \mod p
        \end{align*}
        and the assertion of (2) follows.
\end{proof}

\begin{theorem}\label{multi_p_conv}
    Let $f(t_1,...,t_d)\in \Z[t_1,...,t_d]$.
    Then the $d$-tuple multiple sequence $(r_{n_1,...,n_d}(f))$ and its non-$p$-parts $(r_{n_1,...,n_d}(f)_{\text{non-}p})$ converge in $\Z_p$.
    Moreover the limit value is $0$ if and only if $f(1,...,1)\equiv 0 \mod p$.
\end{theorem}
\begin{proof}
    If there exist $n_1, ..., n_d$ such that $r_{n_1,...,n_d}(f)=0$ then we see that $r_{m_1,...,m_d}(f)=0$ for all $m_i>n_i$ ($i=1,...,d$).
    Thus we suppose that $r_{n_1,...,n_d}(f)\neq0$ for any $n_1, ..., n_d$ throughout this proof.

    First we show that $\lim_{n_1,...,n_d\to \infty}r_{n_1,...,n_d}(f)=0$ if and only if $f(1,...,1)\equiv 0 \mod p$.
    If $\lim_{n_1,...,n_d\to\infty}r_{n_1,...,n_d}(f)=0$ then there are $n_1, ..., n_d>0$ such that $r_{n_1,...,n_d}(f)\equiv 0 \mod p$.
    By (2) of \cref{multi_0_conditins}, we have $f(1,...,1)\equiv 0 \mod p$.
    Conversely, we suppose that $f(1,...,1)\equiv 0 \mod p$.
    For each $i=1,...,d$ and each $n_1,...,n_{i-1}, n_{i+1},...,n_d \in \Z_{>0}$, we set
    \begin{align*}
        &g_{n_1,...,n_{i-1},n_{i+1},...,n_d}(f)(t_i) \\
        &=\Res(t_1^{p^{n_1}}-1,...,\Res(t_{i-1}^{p^{n_{i-1}}}-1,\Res(t_{i+1}^{p^{n_{i+1}}}-1,...,\Res(t_d^{p^{n_d}}-1,f)))).
    \end{align*}
    We note that $r_{n_1,...,n_d}(f)=\Res(t_i^{p^{n_i}}-1,g_{n_1,...,n_{i-1},n_{i+1},...,n_d}(f))$.
    Then for each $i=1,...,d$ and $n_i>1$ we have
    \begin{align*}
        \frac{r_{n_1,...,n_i,...,n_d}(f)}{r_{n_1,...,n_i-1,...,n_d}(f)}=\N_{\Q(\zeta_{p^{n_i}})/\Q}(g_{n_1,...,n_{i-1},n_{i+1},...,n_d}(f)(\zeta_{p^{n_i}})).
    \end{align*}
    Here we claim that $g_{n_1,...,n_{i-1},n_{i+1},...,n_d}(f)(\zeta_{p^{n_i}})$ is not relatively prime to $p$ in $\Q(\zeta_{p^{n_i}})$.
    Indeed, we have
    \begin{align*}
        &g_{n_1,...,n_{i-1},n_{i+1},...,n_d}(f)(\zeta_{p^{n_i}}) \\
        &=\prod_{\zeta_1^{p^{n_1}}=1}\cdots\prod_{\zeta_{i-1}^{p^{n_{i-1}}}=1}\prod_{\zeta_{i+1}^{p^{n_{i+1}}}=1}\cdots\prod_{\zeta_d^{p^{n_d}}=1}f(\zeta_1,...,\zeta_{{i-1}},\zeta_{p^{n_i}},\zeta_{{i+1}},...,\zeta_d).
    \end{align*}
    By using the assumption $f(1,...,1)\equiv 0 \mod p$ and (1) of \cref{multi_0_conditins}, we see that each $f(\zeta_1,...,\zeta_d)$ in the above product is not relatively prime to $p$, and the assertion of the claim follows.
    Thus we have $\N_{\Q(\zeta_{p^{n_i}})/\Q}(g_{n_1,...,n_{i-1},n_{i+1},...,n_d}(\zeta_{p^{n_i}}))\equiv 0 \mod p$.
    This argument holds for each tuple of $(n_1,...,n_d)$.
    Therefore, the $p$-parts of $r_{n_1,...,n_d}(f)$ strictly increase when $(n_1,...,n_d)$ increase, and we have $\lim_{n_1,...,n_d\to\infty }r_{n_1...,n_d}(f)=0$.

    Next we suppose that $f(1,..,1)\not \equiv 0 \mod p$.
    We note that $r_{n_1,...,n_d}(f) \not \equiv 0 \mod p$ for any $n_1, ..., n_d>0$ by (2) of \cref{multi_0_conditins}.
    We show that the sequence $(r_{n_1,n_2,...,n_d}(f))$ is a Cauchy sequence in $\Z_p$.
    By using the Artin's reciprocity law in the same way as the proof of \cref{p_limits} for $g_{n_2,...,n_d}(t_1)$, we have $r_{N_1,n_2,...,n_d}(f)\equiv r_{n_1,n_2,...,n_d}(f)\mod{p^{n_1}}$ for all $N_1>n_1$.
    By repeating this process several times, we obtain that
    \begin{align}\label{multi_equiv}
        r_{N_1,N_2,...,N_d}(f)\equiv r_{n_1,n_2,...,n_d}(f)\mod{p^{\min \{n_1, n_2,...,n_d\}}}        
    \end{align}
    for all $N_i>n_i$ ($i=1,2,...,d$), and the assertion follows.
    
    In the case of $f(1,...,1)\equiv 0 \mod p$, namely the limit value is $0$,
    the non-$p$-parts of $(r_{n_1,...,n_d}(f))$ also satisfies the equivalence condition of (\ref{multi_equiv}).
    Thus we see that the sequence of the non-$p$-parts $(r_{n_1,...,n_d}(f)_{\text{non-}p})$ also converges.
\end{proof}

The proof of \cref{multi_p_conv} yields that the following result holds.
\begin{corollary}\label{sub_prod_seq}
    Let $f\in\Z[t_1,...,t_d]$.
    Fix $n_1,...,n_d$ and let $g_i(t_i)$ be a divisor of $t_i^{p^{n_i}}-1$ for each $i=1,...,d$.
    Then 
    \begin{align*}
        \Res\left(\frac{t_1^{p^{n_1}}-1}{g_1(t_1)},...,\Res\left(\frac{t_d^{p^{n_d}}-1}{g_d(t_d)},f\right)\right)
    \end{align*}
    and their non-$p$-parts also converge in $\Z_p$.
\end{corollary}

We also obtain the following result by \cref{multi_sequentially_limit} and the proof of \cref{multi_p_conv}.
\begin{corollary}
    We have
    \begin{align*}
        \lim_{n_1,...,n_d \to \infty}r_{n_1,...,n_d}(f)=\lim_{n_{\s(1)}\to \infty}\cdots \lim_{n_{\s(d)}\to \infty}r_{n_1,...,n_d}(f)
    \end{align*}
    for all $f\in \Z[t_1,...,t_d]$ and $\s \in \sym_d$.
\end{corollary}

For use in \cref{link_homology}, we give a result for the signatures of cyclic resultants.
This is a generalization of \cite{UekiYoshizaki-plimits}*{Lemma 5.1}.
\begin{proposition}\label{multi_sgn_prop}
    Let $0\neq f(t_1,,...,t_d) \in\Z[t_1,...,t_d]$ and suppose that $r_{n_1,...,n_d}(f)\neq0$ for any $n_1,...,n_d>0$.
    If $p\neq2$, then we have $r_{n_1,...,n_d}(f)>0$ if and only if $f(1,...,1)>0$.
    If $p=2$, then we have $r_{n_1,...,n_d}(f)>0$ if and only if $r_{1,...,1}(f)>0$.
\end{proposition}
\begin{proof}
    Fix $n_1,...,n_d$.
    As in the proof of \cref{multi_p_conv}, we have
    \begin{align*}
        r_{n_1,...,n_d}(f)=\N_{\Q(\zeta_{p^{n_i}})/\Q}(g_{n_1,...,n_{n-1},n_{i+1},...,n_d}(f)(\zeta_{p^{n_i}}))r_{n_1,...,n_{i}-1,...,n_d}(f).
    \end{align*}
    If $\Q(\zeta_{p^{n_i}})$ is a totally imaginary field, then the norm is positive
    and we have $\sgn(r_{n_1,...,n_d}(f))=\sgn(r_{n_1,...,n_{i}-1,...,n_d}(f))$.
    We see that $\Q(\zeta_{p^{n}})$ is a totally imaginary field except for $(p,n)=(2,1)$.
    Thus the signature is determined by $\Res(t_1-1,...,\Res(t_d-1,f))=f(1,...,1)$ for $p\neq2$ and $r_{1,...,1}(f)$ for $p=2$.
\end{proof}

Here we give an example of the $p$-adic limits for somewhat general polynomials.

\begin{example}\label{example}
    Let $g \in \Z[t_2,...,t_d]$ and $f=at_1^n+g \in \Z[t_1,...,t_d]$.
    We assume that $a\neq0$ and $n\geq1$.
    Then we have
    \begin{align*}
        \lim_{n_1,...,n_d\to\infty}r_{n_1,...,n_d}(f)=
        \begin{cases}
            0 & \text{if }p\mid f(1,...,1) \\
            \omega_p(f(1,...,1)) & \text{if }p\nmid f(1,...,1) \text{ and } p\neq2 \\
            \omega_p(g(1,...,1)-a) & \text{if }p\nmid f(1,...,1) \text{ and } p=2,
        \end{cases}
    \end{align*}
    where $\omega_p(m)$ denote the Teichm\"{u}ller character of $m$ for $p$.
\end{example}
\begin{proof}
    We suppose that $n_1>n$.
    Then we have
    \begin{align}\label{factorizationEq}
        \prod_{\zeta^{p^{n_1}}=1}(t-\zeta^n)=\left(t^{p^{n_1-v_p(n)}}-1\right)^{v_p(n)},
    \end{align}
    where $v_p$ denote the $p$-adic valuation.
    We also have
    \begin{align*}
        \Res(t_1^{p^{n_1}}-1,f)&=\prod_{\zeta_1^{p^{n_1}}=1}(g+a\zeta_1^n) \\
        &=(-a)^{p^{n_1}}\prod_{\zeta_1^{p^{n_1}}=1}(-g/a-\zeta_1^n).
    \end{align*}
    By using (\ref{factorizationEq}), we obtain
    \begin{align*}
        \Res(t_1^{p^{n_1}}-1,f)&=(-a)^{p^{n_1}}\prod_{\zeta_1^{p^{n_1}}=1}((-g/a)^{p^{n_1-v_p(n)}}-1)^{p^{v_p(n)}} \\
        &=((g)^{p^{n_1-v_p(n)}}-(-a)^{p^{n_1-v_p(n)}})^{p^{v_p(n)}}.
    \end{align*}
    Thus we have
    \begin{align*}
        r_{n_1,...,n_d}(f)=\prod_{\zeta_2^{p^{n_2}}=1}\cdots\prod_{\zeta_d^{p^{n_d}}=1}(g(\zeta_2,...,\zeta_d)^{p^{n_1-v_p(n)}}\mp a^{p^{n_1-v_p(n)}})^{p^{v_p(n)}},
    \end{align*}
    where the signature $\mp$ is $-$ if $p=2$ and $+$ if $p$ is an odd prime number.
    By (1) of \cref{multi_0_conditins}, we see that $|g(\zeta_2,...,\zeta_d)-g(1,...,1)|_p<1$ and we have
    \begin{align*}
        \lim_{n_1\to\infty}g(\zeta_2,...,\zeta_d)^{p^{n_1}}&=\lim_{n_1\to\infty}g(1,...,1)^{p^{n_1}} \\
        &=\omega_p(g(1,...,1)).
    \end{align*}
    Then we obtain
    \begin{align*}
        \lim_{n_1,..,n_d\to\infty}r_{n_1,...,n_d}(f)&=\lim_{n_2,...,n_d\to\infty}\prod_{\zeta_2^{p^{n_2}}=1}\cdots\prod_{\zeta_d^{p^{n_d}}=1}(\omega_p(g(1,...,1))\mp \omega_p(a))^{p^{v_p(n)}} \\
        &=\lim_{n_2,...,n_d\to\infty}(\omega_p(g(1,...,1))\mp \omega_p(a))^{p^{n_2+\cdots+n_d+v_p(n)}}.
    \end{align*}
    Since $\omega_p(\omega_p(g(1,...,1))\mp \omega_p(a))=\omega_p(g(1,...,1)\mp a)$ the assertion follows.
\end{proof}

%--------------------------------------------------------------------------------------
\section{The $p$-adic limits of the sizes of the homology groups}\label{link_homology}
Let $M$ be a closed integral homology $3$-sphere and $L=l_1\cup...\cup l_d\subset M$ be a tame link with $d$ components.
Let $\Delta_L(t_1,...,t_d)$ denote the multivariable Alexander polynomial of $L$.
The size of the first homology group of an abelian covering of $M$ branched along $L$ can be calculated using iterated cyclic resultants of the Alexander polynomials of sublinks of $L$.
This section gives the $p$-adic convergence of the sizes of the first homology groups in $\Z_p^d$-coverings as an application of \cref{multi_p_conv} and \cref{sub_prod_seq}.
We also provide concrete examples for $\Z_p^2$-coverings branched along the twisted Whitehead links.

\subsection{The formula of the first homology groups}
We recall the construction of abelian branched coverings of $(M,L)$.
Let $G$ be a finite abelian group and $\pi:\pi_1(M-L)\to G$ be a surjective homomorphism.
Let $M_{\pi}\to M$ denote the branched cover corresponding to $\pi$, that is, the Fox completion of the unbranched covering corresponding to $\ker\pi$.
Then the transformation group of $M_{\pi}\to M$ is $G$.
Note that any finite abelian cover of $M$ branched along $L$ is obtained in this way.

We also recall the formula for the size of the first homology groups of abelian branched coverings given by Mayberry--Murasugi and Porti.
Let $\hat{G}$ denote the set of homomorphisms from $G$ to $\C^*$.
We choose meridians $m_1,...,m_d \in \pi_1(M-L;\Z)$.
For $\xi \in \hat{G}$, we define a sublink $L_{\xi}$ of $L$ as
\begin{align*}
    L_{\xi}=\cup_{\xi(\pi(m_i))\neq1}l_i.
\end{align*}
We write $L_{\xi}=l_{i_1}\cup \cdots l_{i_k}$.
For the trivial homomorphism $1\in\hat{G}$, we define $L_{1}=\emptyset$ and $\Delta_{L_1}=1$.
We also set
\begin{align*}
    \hat{G}^{(1)}=\{\xi \in \hat{G}\mid L_{\xi} \text{ has a single component}\}.
\end{align*}
We write $i(\xi)$ for the index of the meridian corresponding to $\xi\in\hat{G}^{(1)}$.
For a group $H$, let $|H|$ be the size of $H$ if $H$ is finite and $0$ if $H$ is infinite.
Then Mayberry--Murasugi and Porti gave the following theorem.
\begin{theorem}[{\cite{Mayberry-Murasugi}*{Theorem 10.1}, \cite{Porti2004}*{Theorem 1.1}}]\label{porti}
    \begin{align}
        |H_1(M_{\pi};\Z)|=\frac{|G|}{\prod_{\xi \in \hat{G}^{(1)}}|1-\xi(\pi(m_{i(\xi)}))|}\prod_{\xi \in \hat{G}}|\Delta_{L_{\xi}}(\xi(\pi(m_{i_1})),...,\xi(\pi(m_{i_k})))|.
    \end{align}
\end{theorem}

\subsection{$\Z_p^d$-coverings}\label{Zpdcovering}
Let $\pi:\pi_1(M-L)\to \Z_p^d$ be the composition map of the abelianization $\pi_1(M-L) \to \pi_1(M-L)^{ab}$ ($\cong \Z^d$) and the inclusion $\Z^d \hookrightarrow \Z_p^d$.
For a finite index subgroup $\Gamma \subset \Z_p^d$ let $\pi_{\Gamma}:\pi_1(M-L)\to \Z_p^d/\Gamma$ denote the composition of $\pi$ and the canonical surjective map $\Z_p^d \to \Z_p^d/\Gamma$.
Let $M_{\Gamma}\to M$ denote the branched covering of $(M,L)$ corresponding to $\pi_{\Gamma}$. 
\begin{definition}\label{defiOfZpd}
    A branched $\Z_p^d$-covering of $(M,L)$ is the inverse system of branched coverings $M_{\Gamma} \to M$ where $\Gamma$ runs finite index open subgroups in $\Z_p^{\,d}$.
\end{definition}
We obtain the $p$-adic convergence of the non-$p$-parts of the sizes of the first homology groups in branched $\Z_p^d$-covering of $(M,L)$.
For a $\Z_p$-module $A$ let $\rank A=\dim_{\F_p}A\otimes \F_p$.
\begin{lemma}\label{keyLemForHomConv}
    Let $\Gamma_1\supset \Gamma_2 \supset \cdots$ be a descending sequence of finite open subgroups such that $\cap_n\Gamma_n=\{0\}$ in $\Z_p^{\,d}$.
    Then we have
    \begin{itemize}
        \item[(1)] for all sufficiently large $n$, $\rank \Z_p^{\,d}/\Gamma_n=d$.
        \item[(2)] for each meridian $m$ of $L$, $\lim_{n\to\infty}\ord(\pi_{\Gamma_n}(m))=\infty$.
    \end{itemize}
\end{lemma}
\begin{proof}
    We put $G_n=\Z_p^{\,d}/\Gamma_n$.

    (1) We note that $\varprojlim_nG_n\cong\Z_p^{\,d}$ by the assumption $\cap_n\Gamma_n=\{0\}$.
    Since the canonical homomorphisms $F_n:\Z_p^{\,d}\to G_n$ and $f_{n/n-1}:G_n \to G_{n-1}$ are surjective the induced homomorphisms
    $\tilde F_n:\Z_p^{\,d}\otimes \F_p \to G_n\otimes \F_p$ and $\tilde f_{n/n-1}:G_n\otimes \F_p \to G_{n-1}\otimes \F_p$ are also surjective.
    Thus we have that the sequence $(\rank G_n)$ is monotonically increasing.
    We suppose that $\rank G_n<d$ for all $n>0$.
    Then we have that $\dim_{\F_p}\ker \tilde F_n \geq 1$ for all $n$.
    Since $\ker \tilde F_n\subset \ker \tilde F_{n-1}$ for all $n>1$, we see that $\dim_{\F_n}\cap_n\ker \tilde F_n\geq1$.
    Thus there exists a non zero element $\sum_i a_i\otimes b_i \in \cap_n\ker \tilde F_n$.
    We take $\beta_i\in\Z_p$ such that $\beta_i \equiv b_i \mod p$ for each $i$.
    Then we have that $\sum_i a_i\otimes b_i=\sum_i \beta_ia_i\otimes 1=(\sum_i\beta_ia_i)\otimes 1$, and we obtain $\sum_i\beta_ia_i\in \cap_n\ker F_n=\cap_n\Gamma_n$.
    This is a contradiction.
    
    (2) We note that the sequence $(\ord(\pi_{\Gamma_n}(m)))$ is monotonically increasing.
    We suppose that there exists $N>$ such that $\ord(\pi_{\Gamma_n}(m))<N$ for all $n>0$.
    Then the element $(\pi_{\Gamma_n}(m))_n$ in $\varprojlim_n\Z_p^{\,d}/\Gamma_n$ is torsion.
    This contradicts the fact that $\varprojlim_n\Z_p^{\,d}/\Gamma_n\cong\Z_p^{\,d}$.
\end{proof}
For $f\in\Z[t_1,...,t_d]$ and positive integers $n_1,...,n_d$, we set
\begin{align*}
    r'_{n_1,...,n_d}(f)=\Res\left(\frac{t_1^{p^{n_1}}-1}{t_1-1},...,\Res\left(\frac{t_d^{p^{n_d}}-1}{t_d-1},f\right)\right).
\end{align*}
\begin{theorem}\label{convergenceHom}
    Let $M$ be a closed integral homology $3$-sphere and $L$ be a link in $M$.
    % Let $L\subset M$ be a link and $(M_n \to M)$ be a $\Z_p^d$-covering of $M$ branched along $L$. 
    Suppose that for each finite index open subgroup $\Gamma\subset \Z_p^{\,d}$, $M_{\Gamma}$ is a rational homology $3$-sphere, that is, $H_1(M_{\Gamma};\Z)$ is a finite group.
    Then for each descending sequence of finite open subgroups $\Gamma_1\supset\Gamma_2\supset \cdots$ in $\Z_p^{\,d}$ such that $\cap_{n}\Gamma_n=\{0\}$, the non-$p$-parts of $|H_1(M_{\Gamma_n};\Z)|$ converge in $\Z_p$.
    Moreover, the limit value does not depend on the choice of a sequence of finite open subgroups.
\end{theorem}
\begin{proof}
    We fix a descending sequence of finite index open subgroups $\Gamma_1\supset\Gamma_2\supset \cdots$ in $\Z_p^{\,d}$ such that $\cap_{n}\Gamma_n=\{0\}$.
    We put $G_n=\Z_p^{\,d}/\Gamma_n$.
    By using \cref{keyLemForHomConv} (1) we have $\rank G_n=d$ for all sufficiently large $n$.
    We fix such $n$.
    Then there exists a tuple of positive integers $(n_1,...,n_d)$ such that $G_n\cong \oplus_i \Z/p^{n_i}\Z$.
    Thus the size of the first homology group $|H_1(M_{\Gamma_n};\Z)|$ is
    \begin{equation}\label{nthLayerH1Size}
        \frac{p^{n_1+\cdots+n_d}}{\prod_{\xi \in \hat{G}_n^{(1)}}|1-\xi(\pi_{\Gamma_n}(m_{i(\xi)}))|}\prod_{\xi \in \hat{G}_n}|\Delta_{L_{\xi}}(\xi(\pi_{\Gamma_n}(m_{i_1})),...,\xi(\pi_{\Gamma_n}(m_{i_k})))|
    \end{equation}
    by \cref{porti}.
    We note that the fraction part of (\ref{nthLayerH1Size}) is $p$-power integer, thus it does not need to be taken into account for non-$p$ part calculation.
    Thus we have
    \begin{equation}\label{nthLayerHom}
        |H_1(M_{\Gamma_n};\Z)_{\text{non-}p}|=\prod_{\xi \in \hat{G}_n}|\Delta_{L_{\xi}}(\xi(\pi_{\Gamma_n}(m_{i_1})),...,\xi(\pi_{\Gamma_n}(m_{i_k})))|_{\text{non-}p}.
    \end{equation}
    % We fix a tuple $(n_1,...,n_d)$ and let $G=\Z_p^{\,d}/\Gamma_{n_1,...,n_d}\cong \bigoplus_i\Z/p^{n_i}\Z$.
    % For a meridian $m$ let $\ord(m)$ denote the order of $\pi_{\Gamma_n}(m)$ in $G$,
    % For each meridian $m$ we set $\hat{G_n}^{(1)}_m=\{\xi \in \hat{G_n}^{(1)}\mid \xi(\pi_{Gamma_n}(m))\neq1\}$.
    % Then we have
    % \begin{align*}
    %     \prod_{\xi \in \hat{G_n}^{(1)}}|1-\xi(\pi(m_{i(\xi)}))|&=\prod_{i=1}^{d}\prod_{\xi \in \hat{G}^{(1)}_{m_i}}|1-\xi(\pi(m_i))| \\
    %     &=\prod_{i=1}^{d}\ord(m_i).
    % \end{align*}
    % We note that $\ord(m_i)=p^{n_i}$ for each meridian $m_i$.
    % Indeed, the abelianization induces a map $\pi_{\Gamma_{n_1,...,n_d}}:H_1(E(L);\Z)\to G$.
    % Since $M$ is an integral homology $3$-sphere, we have $H_1(E(L);\Z)\cong \Z^d$ .
    % Thus $G=(\Z/p^n\Z)^d$ is generated by the images of meridians, and the assertion follows.
    % Then we have that $\prod_{\xi \in \hat{G}^{(1)}}|1-\xi(\pi(m_{i(\xi)}))|=|G|$ and
    For a sublink $L'$ in $L$ let $d(L')$ denote the number of components of $L'$ and let $i(L')_1,...,i(L')_{d(L')}\in\{1,...,d\}$ denote the indices corresponding the meridians of $L'$.
    We set $\hat{G}_n(L')=\{\xi\in\hat{G}_n\mid L_{\xi}=L'\}$.
    Then the right hand side of (\ref{nthLayerHom}) is
    \begin{align*}
        \prod_{L'\subset L}\prod_{\xi \in \hat{G}_n(L')}|\Delta_{L'}(\xi(\pi_{\Gamma_n}(m_{i(L')_1})),...,\xi(\pi_{\Gamma_n}(m_{i(L')_{d(L')}})))|_{\text{non-}p}.
    \end{align*}
    Here we note that, for each $m_{i(L')_k}$,  $\xi(\pi_{\Gamma_n}(m_{i(L')_k}))$ runs $\ord(\pi_{\Gamma_n}(m_{i(L')_k}))$-th roots of unities except for $1$.
    Thus, by putting $N_n(m)=v_p(\ord(\pi_{\Gamma_n}(m)))$ for each meridian $m$, we have
    \begin{align*}
        \prod_{\xi \in \hat{G}_n(L')}&|\Delta_{L'}(\xi(\pi_{\Gamma_n}(m_{i(L')_1})),...,\xi(\pi_{\Gamma_n}(m_{i(L')_{d(L')}})))| \\
        &=|r'_{N_n(m_{i(L')_1}),...,N_n(m_{i(L')_{d(L')}})}(\Delta_{L'}(t_1,...,t_{d(L')}))|.
    \end{align*}
    By using \cref{keyLemForHomConv} (2), we have $\lim_{n\to\infty}N_n(m)=\infty$.
    Therefore, the assertion follows from \cref{sub_prod_seq}.
    The independence of the limit from the choice of the sequence $(\Gamma_n)$ is obvious.
\end{proof}
\begin{remark}\label{multipleRem}
    Let $\Gamma_{n_1,...,n_d}=\oplus_i p^{n_i}\Z_p$ for each tuple of positive integers $(n_1,...,n_d)$.
    Then, by \cref{sub_prod_seq} and the proof of \cref{convergenceHom}, the multiple sequence of the non-$p$-parts of $|H_1(M_{\Gamma_{n_1,...,n_d}};\Z)|$ converges at $\Z_p$ and the limit value coincides with the one in \cref{convergenceHom}.
\end{remark}
\begin{remark}\label{sublink_sgn_rem}
    By the proof of \cref{multi_sgn_prop} we have
    \begin{align*}
        \sgn(r'_{n_1,...,n_d}(f))=
        \begin{cases}
            1 & \text{if }p\neq2, \\
            \sgn(f(-1,...,-1)) & \text{if }p=2.
        \end{cases}
    \end{align*}
\end{remark}
% \begin{remark}\label{multiToSingle}
%     For each positive integer $n$ let $\Gamma_n=(p^n\Z_p)^d$.
%     Then, by the proof of \cref{CauchyConverge}, we see that
%     \begin{equation*}
%         \lim_{n_1,...,n_d\to\infty}|H_1(M_{n_1,...,n_d};\Z)|_{\text{non-}p}=\lim_{n\to\infty}|H_1(M_n;\Z)|_{\text{non-}p}.
%     \end{equation*}
% \end{remark}
Let $h(M,L)$ denote the limit value in \cref{convergenceHom}.
\subsection{The $p$-adic torsion}
The $p$-adic limit value $h(M,L)$ in \cref{convergenceHom} coincides with the topological invariant, the (relative) $p$-adic torsion defined by Kinoke~\cite{Kionke2020}.
In this section, we show the coincidence.
For a module $M$, let $\text{tors}M$ denote the torsion part of $M$.
Let $X$ be a finite CW-complex and $\widetilde{X}$ denote the universal covering of $X$.
% Let $\Gamma_X$ denote the fundamental group of $X$.
% Then the fundamental group $\pi_1(X)$ acts on $\widetilde{X}$ and
For a finite index open subgroup $\Gamma \subset \pi_1(X)$, $\widetilde{X}/\Gamma\to X$ denote the unbranched covering space corresponding to $\Gamma$ by Galois theory.
We note that $\widetilde{X}/\Gamma$ is a finite sheeted covering of $X$ whose covering transformation group is isomorphic to $\pi_1(X)/\Gamma$.
Let $G$ be a profinite group which has an open pro-$p$ subgroup.
A group homomorphism $\phi:\pi_1(X) \to G$ is called a {\it virtual pro-$p$ completion} of $\pi_1(X)$.
For a subset $A$ in $X$ and a commutative ring $R$ in which $p$ is invertible, the relative $p$-adic torsion is defined by
\[
    t_j^{[p]}(X,A;\phi,R)=\#_p^{G}(\text{tors}(\varinjlim_{N<G}H^j(\widetilde{X}/\phi^{-1}(N),A_N;R)))
\]
where $N$ run open normal subgroups of $G$, $A_N$ is the inverse image of $A$ by the covering map $\widetilde{X}/\phi^{-1}(N)\to X$, and $\#_p^{G}$ is the $p$-adic cardinality function which is defined in \cite[Section 4]{Kionke2020}.

\begin{proposition}\label{p_adic_torsion}
    We assume the same assumption as in \cref{convergenceHom}.
    Furthermore, we assume that $\Delta_L$ does not vanish on $p$-power-th roots of unities except for $(1,...,1)$.
    Let $X=E(L)$, $A=\partial E(L)$ be the boundary of $E(L)$, $R=\Z[1/p]$, and $\phi$ be the composition of the abelianization $\pi_1(E(L))\to \pi_1(E(L))^{ab}\cong \Z^d$ and the inclusion $\Z^d \subset \Z_p^d$.
    Then we have
    \[
        h(M,L)=t_2^{[p]}(E(L),\partial E(L);\phi,\Z[1/p]).
    \]
\end{proposition}
\begin{proof}
    Let $\Gamma_n=(p^n\Z_p)^d$ and $M_n=M_{\Gamma_n}$ for each positive integer $n$.
    Since $h(M,L)=\lim_{n\to\infty}|H_1(M_n,\Z)|_{\text{non-}p}$ it suffices to show that
    \begin{equation*}
        \lim_{n\to\infty}|H_1(M_n,\Z)_{\text{non-}p}|=t_2^{[p]}(E(L),\partial E(L);\phi,\Z[1/p]).
    \end{equation*}
    Set $E(L)_n=\widetilde{E(L)}/\phi^{-1}(\Gamma_n)$.
    In this setting, Kionke~\cite[Theorem 5.11]{Kionke2020} show that
    \begin{align*}
        \lim_{n\to \infty}|H^2(E(L)_n,\partial E(L)_n;\Z[1/p])|=t_2^{[p]}(E(L),\partial E(L);\phi ,\Z[1/p]).
    \end{align*}
    We note that $|\text{tors}H^2(E(L)_n,\partial E(L)_n;\Z[1/p])|=|\text{tors}H^2(E(L)_n,\partial E(L)_n;\Z)_{\text{non-}p}|$.
    By considering the Lefschetz duality, we have
    \begin{align*}
        H^2(E(L)_n,\partial E(L)_n;\Z)\cong H_1(E(L)_n;\Z).
    \end{align*}
    Let $h_n:M_n \to M$ be the covering map.
    We note that the exterior of $h_n^{-1}(L)$ in $M_n$ is homeomorphic to $E(L)_n$.
    We claim that $|\text{tors}H_1(E(L);\Z)_{\text{non-}p}|=|H_1(M_n;\Z)_{\text{non-}p}|$.
    By using \cite[Theorem 2.1]{HartleyMurasugi1978} we have
    \begin{equation*}
        \text{tors}H_1(E(L);\Z)\cong H_1(M_n;\Z)/\langle h_n^{-1}(L) \rangle,
    \end{equation*}
    where $\langle h_n^{-1}(L) \rangle$ is a submodule of $H_1(M_n;\Z)$ which is generated by the components of $h_n^{-1}(L)$.
    We recall that $\mu_p=\{\zeta \in \overline{\Q}_p\mid \exists N\in\Z \text{ such that }\zeta^{p^N}=1\}$.
    By the assumption of $\Delta_L(\zeta_1,...,\zeta_d)\neq0$ for $(\zeta_1,...,\zeta_d)\in \mu_p^d\setminus \{(1,...,1)\}$, $L$ does not decompose in $M_n \to M$ (see \cite[Prop 10.4]{TatenoUeki2024}).
    For each component $l$ of $L$, since $[l]=0$ in $H_1(M;\Z)$, $l$ has the Seifert surface $\Sigma_l$ in $M$.
    Thus we have $h_n^{-1}(\Sigma_l)$ is a $2$-chain such that $\partial h_n^{-1}(\Sigma_l)=eh_n^{-1}(l)$ where $e$ denote the ramification index of $l$ in $M_n\to M$.
    Since $e$ divides $|G|=p^{nd}$ we see that $[h_n^{-1}(l)]$ is in the $p$-part of $H_1(M_n,\Z)$.
    Therefore, we have $\text{tors}H_1(E(L);\Z)_{\text{non-}p}\cong (H_1(M_n;\Z)/\langle h_n^{-1}(L) \rangle)_{\text{non-}p}$, and the assertion follows.
\end{proof}
\begin{remark}
    In the case where $L$ is a knot, that is studied in \cite{UekiYoshizaki-plimits}, the $p$-adic limit of the first homologies coincides with the $p$-adic torsion $t_2^{[p]}(E(L),\partial E(L);\phi,\Z[1/p])$ since $\Delta_L(1)=\pm 1$.
\end{remark}
\subsection{Twisted Whitehead links}
Let $k$ be an integer and $L_k$ denote the $k$-twisted Whitehead link (\cref{L_k}).
\begin{figure}[h]
    \centering
    \includegraphics[width=110mm]{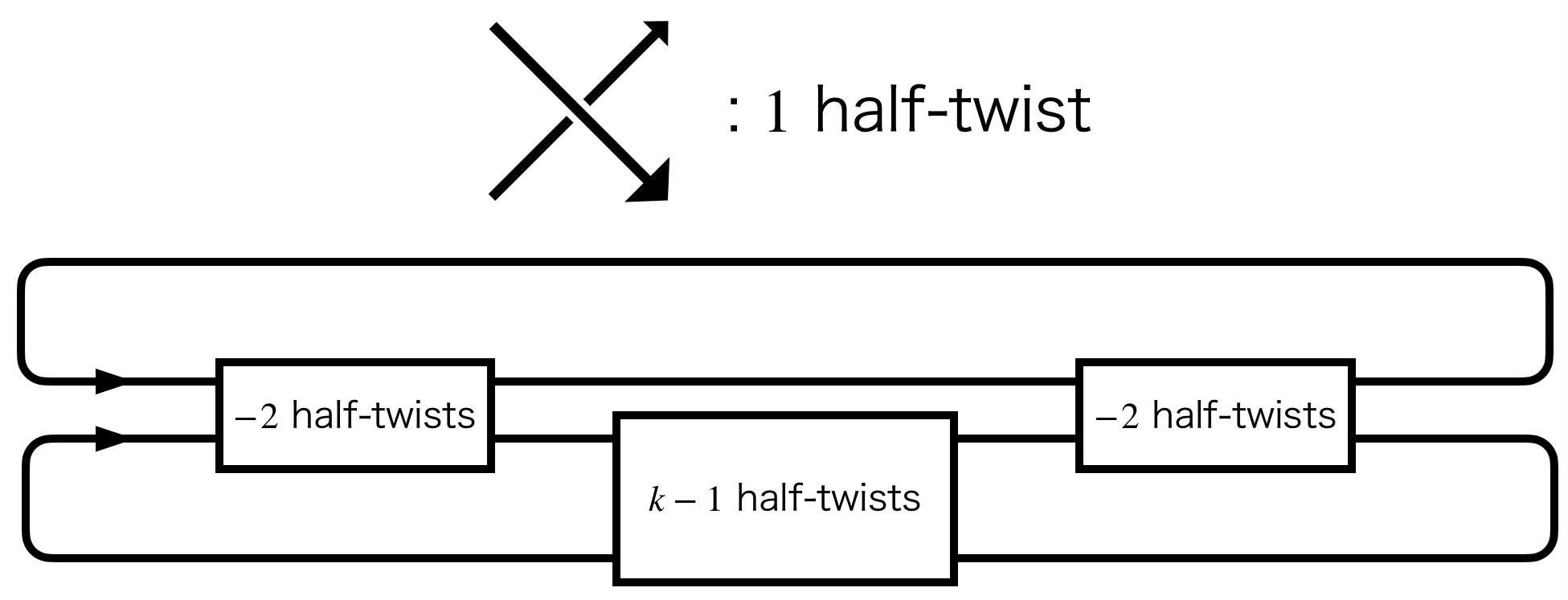}
    \caption{$L_k$}\label{L_k}
\end{figure}
%(see \cite{TatenoUeki2024}*{Definition 12.1} for the definition of the Whitehead links).
Then we have the following.
\begin{example}\label{twistedWhiteheadLink}
    We have
    \begin{align*}
        &h(S^3,L_k)= \\
        &\begin{cases}
            \frac{m|m|_p}{\omega_p(m|m|_p)} & \text{if }k=2m, \\
            \frac{\omega_p(2)}{2} & \text{if }k=2m+1 \text{ and } p\neq2, \\
             (-1)^m\frac{\omega_2(m+1)-\omega_2(m)}{k}\prod_{\zeta \in \mu_2\setminus \{\pm1\}}2^{-\nu_{\zeta}}\log\frac{m\zeta+m+1}{m\zeta+m+\zeta} & \text{if }k=2m+1 \text{ and } p=2,
        \end{cases}
    \end{align*}
    where we set $\omega_2(m)=0$ for even integer $m$ and $\nu_{\zeta}=v_2(\log((m\zeta+m+1)/(m\zeta+m+\zeta)))$ for each $\zeta \in \mu_2$.
\end{example}
\begin{proof}
    Let $\Gamma_{n_1,n_2}=p^{n_1}\Z_p\oplus p^{n_2}\Z_p$ and $S^3_{n_1,n_2}=S^3_{\Gamma_{n_1,n_2}}$ for each positive integers $n_1$ and $n_2$.
    Then, by \cref{multipleRem}, we have $\lim_{n_1,n_2\to \infty}|H_1(S^3_{n_1,n_2};\Z)_{\text{non-}p}|=h(S^3,L_k)$.
    The Alexander polynomial of $L_k$ is of the form
    \begin{align*}
        \Delta_k(t_1,t_2)=
        \begin{cases}
            1+m-m(t_1+t_2)+(1+m)t_1t_2 & \text{if } k=2m+1, m\geq0, \\
            m(1+t_1t_2-t_1-t_2) & \text{if } k=2m, m\geq1,
        \end{cases}
    \end{align*}
    as in \cite{TatenoUeki2024}*{Proposition 12.2}.
    We note that the Alexander polynomial of each proper sublink $L' \subset L$ is trivial $\Delta_{L'}=1$.
    We also note that $\Delta_k(-1,-1)>0$ for any $k$.
    Thus we have
    \begin{align*}
        |H_1(S^3_{n_1,n_2};\Z)|=r'_{n_1,n_2}(\Delta_k)
    \end{align*}
    by the proof of \cref{convergenceHom} and \cref{sublink_sgn_rem}.
    % Since $\lim_{n\to \infty}r'_{n,n}(\Delta_k)=\lim_{n_1,n_2\to\infty}r'_{n_1,n_2}(\Delta_k)$,
    % we calculate $\lim_{n_1,n_2\to \infty}r'_{n_1,n_2}(\Delta_k)$.
    Set $W(p^n)=\{\zeta \in \C \mid \zeta^{p^n}=1,\ \zeta\neq1\}$ for each prime number $p$ and positive integer $n$.
    
    Firstly, we show the case of $k=2m$.
    We have
    \begin{align*}
        r'_{n_1,n_2}(\Delta_k)&=\prod_{\zeta_1\in W(p^{n_1})}\prod_{\zeta_2 \in W(p^{n_2})}|m(1-\zeta_1)(1-\zeta_2)| \\
        &=\prod_{\zeta_1\in W(p^{n_1})}m^{p^{n_2}-1}|(1-\zeta_1)|^{p^{n_2}-1}\prod_{\zeta_2 \in W(p^{n_2})}|(1-\zeta_2)| \\
        &=m^{(p^{n_1}-1)(p^{n_2}-1)}\prod_{\zeta_1\in W(p^{n_1})}|(1-\zeta_1)|^{p^{n_2}-1}p^{n_2} \\
        &=m^{(p^{n_1}-1)(p^{n_2}-1)}p^{n_1(p^{n_2}-1)+n_2(p^{n_1}-1)}.
    \end{align*}
    Thus we see that $\lim_{n_1,n_2\to\infty} |H_1(S^3_{n_1,n_2};\Z)|=0$,
    and the limit of the non-$p$-parts is $\lim_{n\to\infty} |H_1(S^3_{n_1,n_2};\Z)_{\text{non-}p}|=\lim_{n_1,n_2\to \infty}(m|m|_p)^{p^{n_1+n_2}-p^{n_1}-p^{n_2}+1}=m|m|_p/\omega_p(m|m|_p)$.

    Secondly, we show the case of $k=2m+1$.
    We claim that
    \begin{equation}\label{oddCase}
        r'_{n_1,n_2}(\Delta_k)=\Res\left(\frac{t_1^{p^{n_1}}-1}{t_1-1},\frac{(1-mt_1+m)^{p^{n_2}}-(-1)^{p}(t_1+mt_1-m)^{p^{n_2}}}{t_1+1}\right).
    \end{equation}
    Indeed, set $f(t_1)=(1-mt_1+m)/(t_1+mt_1-m)$.
    Then we have
    \begin{align*}
        r'_{n_1,n_2}(\Delta_k)&=\Res\left(\frac{t_1^{p^{n_1}}-1}{t_1-1},(t_1+mt_1-m)^{p^{n_2}-1}\Res\left(\frac{t_2^{p^{n_2}}-1}{t_2-1},t_2+f(t_1) \right)\right) \\
        &=\Res\left(\frac{t_1^{p^{n_1}}-1}{t_1-1},(t_1+mt_1-m)^{p^{n_2}-1}\left(\frac{f(t_1)^{p^{n_2}}-(-1)^{p}}{f(t_1)+1}\right)\right) \\
        &=\Res\left(\frac{t_1^{p^{n_1}}-1}{t_1-1},\frac{(1-mt_1+m)^{p^{n_2}}-(-1)^{p}(t_1+mt_1-m)^{p^{n_2}}}{t_1+1}\right).
    \end{align*}
    Let $p$ be an odd prime number.
    Since $\Res(f, gh)=\Res(f,g)\Res(f,h)$ for $f,g,h\in \Z[t]$, we see that $r'_{n_1,n_2}(\Delta_k)$ is
    \begin{align}\label{fractionOfRes}
        \frac{\Res((t_1^{p^{n_1}}-1)/(t_1-1),(1-mt_1+m)^{p^{n_2}}+(t_1+mt_1-m)^{p^{n_2}})}{\Res((t_1^{p^{n_1}}-1)/(t_1-1),t_1+1)}.
    \end{align}
    Then the denominator of (\ref{fractionOfRes}) is $1$,
    and the numerator of (\ref{fractionOfRes}) is
    \begin{align*}
        &\Res\left(\frac{t_1^{p^{n_1}}-1}{t_1-1},(1-mt_1+m)^{p^{n_2}}+(t_1+mt_1-m)^{p^{n_2}}\right)\\
        &=\prod_{\zeta \in W(p^{n_1})}((1-m\zeta+m)^{p^{n_2}}+(\zeta+m\zeta-m)^{p^{n_2}}).
    \end{align*}
    Let $\zeta$ be a $p^n$-th roots of unity and set $\alpha=1+m(1-\zeta)$ and $\beta=\zeta - m(1-\zeta)$.
    We note that both $\alpha$ and $\beta$ are prime to $p$ since $1-\zeta$ is a prime element over $p$ in $\Q(\zeta)$ except for $\zeta=1$.
    Moreover we see that the Teichm\"{u}ller characters of $\alpha$ and $\beta$ are $\omega_p(\alpha)=\omega_p(\beta)=1$.
    Therefore, the limit value is
    \begin{align*}
        \lim_{n_1,n_2\to \infty}r'_{n_1,n_2}(\Delta_{k})&=\lim_{n_1\to \infty}\lim_{n_2\to \infty}\prod_{\zeta\in W(p^{n_1})}((1-m\zeta+m)^{p^{n_2}}+(\zeta+m\zeta-m)^{p^{n_2}}) \\
        &=\lim_{n_1\to\infty}\prod_{\zeta \in W(p^{n_1})}2 \\
        &=\lim_{n_1\to\infty}2^{p^{n_1}-1} \\
        &=\frac{\omega_p(2)}{2}.
    \end{align*}
    By the proof of \cref{multi_sgn_prop}, $r'_{n_1,n_2}(\Delta_k)>0$ for each $n_1,n_2>0$.
    Thus we have $\lim_{n_1,n_2\to \infty}|H_1(S^3_{n_1,n_2};\Z)|=\omega_p(2)/2$.

    Next we consider the case of $p=2$.
    Since $\Delta_k(1,1)=2$, \cref{multi_sgn_prop} yields that the $2$-adic limit of $r'_{n_1,n_2}(\Delta_k)$ is $0$.
    We calculate the $2$-adic limit of $r'_{n_1,n_2}(\Delta_k)_{\text{non-}2}$.
    By (\ref{oddCase}) we have
    \begin{align*}
        r'_{n_1,n_2}(\Delta_k)=\Res \left(\frac{t_1^{2^{n_1}}-1}{t_1-1},\frac{(1-mt_1+m)^{2^{n_2}}-(t_1+mt_1-m)^{2^{n_2}}}{t_1+1}\right).
    \end{align*}
    Put $F_{n_2}(t)=((1-mt+m)^{2^{n_2}}-(t+mt-m)^{2^{n_2}})/(t+1)$.
    Then the set of all roots of $F_{n_2}$ is
    \begin{align*}
        \left\{\frac{m\zeta+m+1}{m\zeta+m+\zeta}\mid \zeta^{2^{n_2}}=1,\ \zeta\neq-1\right\}.
    \end{align*}
    It is easy to check by induction that the leading term of $F_{n_2}$ is
    \begin{align*}
        \left(-\prod_{i=0}^{n_2-1}(m^{2^i}+(m+1)^{2^i})\right)t^{2^{n_2}-1}.
    \end{align*}
    Thus, by using (\ref{exchange_law}) we have
    \begin{align*}
        &r'_{n_1,n_2}(\Delta_k)=(-1)^{(2^{n_1}-1)(2^{n_2}-1)}\left(-\prod_{i=0}^{n_2-1}(m^{2^i}+(m+1)^{2^i})\right)^{2^{n_1}-1}2^{n_1}\prod_{\substack{F_{n_2}(\alpha)=0, \\ \alpha\neq 1}}\frac{\alpha^{2^{n_1}}-1}{\alpha-1} \\
        &=\left(\prod_{i=0}^{n_2-1}(m^{2^i}+(m+1)^{2^i})\right)^{2^{n_1}-1}2^{n_1}\prod_{\substack{\zeta \in W(2^{n_2}), \\ \zeta\neq -1}}\frac{\left(\frac{m\zeta+m+1}{m\zeta+m+\zeta}\right)^{2^{n_1}}-1}{\frac{m\zeta+m+1}{m\zeta+m+\zeta}-1} \\
        &=\left(\prod_{i=0}^{n_2-1}(m^{2^i}+(m+1)^{2^i})\right)^{2^{n_1}-1}2^{n_1}\prod_{\substack{\zeta \in W(2^{n_2}), \\ \zeta\neq -1}}\left((m\zeta+m+\zeta)\frac{\left(\frac{m\zeta+m+1}{m\zeta+m+\zeta}\right)^{2^{n_1}}-1}{1-\zeta}\right).
    \end{align*}
    Here we note that
    \begin{align*}
        \prod_{\substack{\zeta \in W(2^{n_2}), \\ \zeta\neq -1}}(1-\zeta)=2^{n_2-1}.
    \end{align*}
    Thus we have
    \begin{align*}
        r'_{n_1,n_2}(\Delta_k)=\underset{(1)}{\underline{\left(\prod_{i=0}^{n_2-1}(m^{2^i}+(m+1)^{2^i})\right)^{2^{n_1}-1}}}&2^{n_1-n_2+1}\underset{(2)}{\underline{\left(\prod_{\substack{\zeta \in W(2^{n_2}), \\ \zeta\neq -1}}m\zeta+m+\zeta\right)}} \\
        &\times\underset{(3)}{\underline{\left(\prod_{\substack{\zeta \in W(2^{n_2}), \\ \zeta\neq -1}}\left(\frac{m\zeta+m+1}{m\zeta+m+\zeta}\right)^{2^{n_1}}-1\right)}}
    \end{align*}
    We calculate the $2$-adic limits of non-$2$-parts of (1), (2) and (3).
    For (1), we set
    \begin{align*}
        G_{n_2}(m)=\prod_{i=0}^{n_2-1}(m^{2^i}+(m+1)^{2^i}).
    \end{align*}
    We claim that
    \begin{align*}
        \lim_{n_2\to\infty}G_{n_2}(m)=
        \begin{cases}
            1 & \text{if }2\mid m, \\
            -1 & \text{if }2\nmid m.
        \end{cases}
    \end{align*}
    To show the claim, we set
    \begin{align*}
        G_n^-(m)=\prod_{i=0}^{n-1}(-1)^{m+1}(m^{2^i}-(m+1)^{2^i}).
    \end{align*}
    for each $n>0$.
    By using Euler's theorem we see that $G^-_{n+1}(m)/G^-_{n}(m)\equiv  1 \mod 2^n$.
    Thus we obtain that $G_n^-(m)$ converges in $\Z_2$ for each $m$.
    We have
    \begin{align*}
        G_n(m)G_n^-(m)&=\prod_{i=0}^{n-1}(-1)^{m+1}(m^{2^{i+1}}-(m+1)^{2^{i+1}}) \\
        &=\frac{G^-_{n+1}(m)}{(-1)^{m+1}(m-(m+1))} \\
        &=(-1)^mG^-_{n+1}(m).
    \end{align*}
    Since $\lim_{n\to\infty}G_n^-(m)\neq 0$, the assertion of the claim follows.

    For (2), 
    we have
    \begin{align*}
        \prod_{\substack{\zeta \in W(2^{n_2}), \\ \zeta\neq -1}}(m\zeta+m+\zeta)&=\Res\left(\frac{t^{2^{n_2}}-1}{t^2-1},(m+1)t+m\right) \\
        &=\frac{m^{2^{n_2}}-(m+1)^{2^{n_2}}}{m^2-(m+1)^2}.
    \end{align*}
    Thus we have
    \begin{align*}
        \lim_{n_2\to\infty}\prod_{\substack{\zeta \in W(2^{n_2}), \\ \zeta\neq -1}}(m\zeta+m+\zeta)=\frac{\omega_2(m+1)-\omega_2(m)}{2m+1},
    \end{align*}
    where we set $\omega_2(n)=0$ for even integers $n$.

    For (3),
    By \cite{UekiYoshizaki-plimits}*{Lemma 5.6 (2)}, we have
    \begin{align*}
        \lim_{n\to\infty}\frac{\left(\frac{m\zeta+m+1}{m\zeta+m+\zeta}\right)^{2^n}-1}{2^n}=\log_2 \frac{m\zeta+m+1}{m\zeta+m+\zeta}.
    \end{align*}

    Combining the above calculations we have
    \begin{align*}
        &\lim_{n_1,n_2\to \infty}|H_1(S^3_{n_1,n_2};\Z)_{\text{non-}2}| \\
        &=(-1)^m\frac{\omega_2(m+1)-\omega_2(m)}{2m+1}\prod_{\zeta \in \mu_2\setminus \{\pm1\}}2^{-\nu_{\zeta}}\log \frac{m\zeta+m+1}{m\zeta+m+\zeta}.
    \end{align*}
\end{proof}

\begin{remark}
    By \cref{p_adic_torsion}, we see that
    \begin{equation*}
        h(S^3,L_k)=t_2^{[p]}(E(L_k),\partial E(L_k);\phi,\Z[1/p])
    \end{equation*}
    in the case of $k=2m+1$.
\end{remark}

We note that the similar calculation for the case of $k=2m$ is done in \cite{TatenoUeki2024}*{Example 13.1} to determine the $p$-parts of $|H_1(S^3_n)|$ where we set $S^3_n=S^3_{n,n}$.
For $2$-parts of the case of $k=2m+1$, by calculating as in the proof of \cref{IwasawaTypeFormula}, we have that
\begin{align*}
    |H_1(S^3_n;\Z)|_2^{-1}=2^{n2^n-2n+1+\sum_{\zeta^{2^n}=1, \zeta\neq \pm 1}\nu_{\zeta}}
\end{align*}
for sufficiently large $n$.
This is an example of \cite{TatenoUeki2024}*{Theorem 10.1 (1)} with $\mu=0$.

\section*{Acknowledgements}
The author would like to thank Jun Ueki and Sohei Tateno.
They gave the author great many helpful comments, which contributed to improving the results of this article.
The author also would like to thank Tomokazu Kashio who also gave the author helpful comments.
%--------------------------------------------------------------------------------------

\bibliographystyle{jplain}
\bibliography{References}

\end{document}